\documentclass[11pt]{amsart}

\usepackage{amsthm}
\usepackage{amsxtra}
\usepackage{amssymb}
\usepackage{bbm}
\usepackage{amsfonts}
\usepackage{textcomp}
\usepackage{lmodern}

\newtheorem{theorem}{Theorem}
\numberwithin{theorem}{section}

\newtheorem{lemma}[theorem]{Lemma}

\newtheorem{proposition}[theorem]{Proposition}

\newtheorem{corollary}[theorem]{Corollary}

{\theoremstyle{definition}
\newtheorem{definition}{Definition}
\numberwithin{definition}{section}
}

\numberwithin{equation}{section}
\numberwithin{figure}{section}

\newcommand{\R}{\mathbb R}

\newcommand{\g}{{\mathrm{g}}}

\newcommand{\hf}{{\frac 12}}
\newcommand{\thf}{\textstyle{\frac 12}}

\newcommand{\bigO}{\mathcal{O}}

\newcommand{\bfu}{\mathbf{u}}
\newcommand{\bfy}{\mathbf{y}}
\newcommand{\bfr}{\mathbf{r}}

\newcommand{\one}{{\mathbbm{1}}}
\newcommand{\mL}{{\mathcal L}}
\newcommand{\ts}{\hspace{.06em}}

\DeclareMathOperator{\trace}{tr}
\DeclareMathOperator{\comp}{c}

\DeclareMathOperator{\supp}{supp}

\title [Trace of the heat kernel and regularity of potentials]{On the trace of Schr\"odinger heat kernels and regularity of potentials}
\author[H. Smith]{Hart Smith}
\address{Department of Mathematics, University of Washington, Seattle, WA 98195-4350, USA}
\email{hfsmith@uw.edu}

\thanks{This material is based upon work supported by 
the National Science Foundation under Grant DMS-1500098}

  

\begin{document}

\begin{abstract}
For the Schr\"odinger operator $-\Delta_\g+V$ on a complete Riemannian manifold with real valued potential $V$ of compact support, we establish a sharp equivalence between Sobolev regularity of $V$ and the existence of finite-order asymptotic expansions as $t\rightarrow 0$ of the relative trace of the Schr\"odinger heat kernel. As an application, we generalize a result of S\`a Barreto and Zworski \cite{SaBZw}, concerning the existence of resonances on compact metric perturbations of Euclidean space, to the case of bounded measurable potentials.
\end{abstract}

\maketitle

\section{Introduction and statement of results}
Consider a Schr\"odinger operator $P_V=-\Delta_\g+V$ on a complete Riemannian manifold
$(M,\g)$ of dimension $n$, with $\Delta_\g$ the Laplace-Beltrami operator. We assume $V\in L^\infty_{\comp}(M)$ is real valued, where $L^\infty_{\comp}(M)$ denotes bounded measurable functions of compact support on $M$. We assume the Ricci curvature of $(M,\g)$ is bounded from below to ensure uniqueness of solutions to the heat equation; see \cite{Dod}.  Let $e^{-tP_0}$ denote the heat semigroup on $(M,\g)$, and $e^{-tP_V}$ the heat semigroup for $P_V$, which can be constructed from $e^{-tP_0}$ by iteration (e.g. see \S\ref{sec:heatkernel} of this paper).

For examples of $(M,\g)$ including compact manifolds \cite{BGM}, and Euclidean space \cite{MM}, \cite{CdV}, it is well known that if $V\in C^\infty_{\comp}(M)$ then $e^{-tP_V}-e^{-tP_0}$ is of trace class for $t>0$, and its trace admits a full asymptotic expansion as $t\rightarrow 0$
$$
\trace\bigl(e^{-tP_V}-e^{-tP_0}\bigr) \sim (4\pi t)^{-\frac n2}\sum_{k=1}^\infty a_k\,t^k,\qquad 0<t\le 1.
$$
In this paper we prove a sharp equivalence between the existence of this expansion to finite order, and finite order Sobolev regularity of $V$. In order to ensure the above difference is of trace class when $n\ge 4$ we make an additional assumption \eqref{assump2} on $(M,\g)$, but for $n\le 3$ we prove that it is trace class using only uniqueness of solutions to the heat equation. Our main result is the following.

\begin{theorem}\label{thm:main} Assume that \eqref{assump2} holds if $n\ge 4$.
Suppose that $V\in L^\infty_{\comp}(M)$ is real valued, and that for a given integer $m\ge 0$ it holds that
\begin{equation}\label{eqn:asymptotics}
\trace\bigl(e^{-tP_V}-e^{-tP_0}\bigr)
=(4\pi t)^{-\frac n2}
\Bigl(c_1 t+c_2 t^2+\cdots+c_{m+1}t^{m+1}+r_{m+2}(t)t^{m+2}\Bigr)
\end{equation}
where $|r_{m+2}(t)|\le C$ for $0<t\le 1$. Then $V\in H^m(M)$. Conversely, if $V\in L^\infty_{\comp}\cap H^m(M)$ then \eqref{eqn:asymptotics} holds with $r_{m+2}(t)\in C\bigl([0,1])$, and in particular $\lim_{t\rightarrow 0^+}r_{m+2}(t)$ exists.
\end{theorem}
Here, $H^m(M)$ with $m\ge 0$ denotes the integer order Sobolev spaces on $M$, consisting of functions whose derivatives up to order $m$ belong to $L^2(M)$. We consider only functions supported in a fixed compact set, so the norm on $H^m(M)$ can be defined using a finite collection of coordinate charts.

For $n\ge 4$, to show that $e^{-tP_V}-e^{-tP_0}$ is trace class we will assume trace class bounds for the heat kernel restricted on one side to a compact set. Let $\mL^1$ denote the trace class operators on $L^2(M)$, which form an ideal in the algebra of bounded operators. If $\one_K$ denotes restriction of functions to $K$, then for $K\subset M$ compact we assume that
\begin{equation}\label{assump2}
\|\one_K\circ e^{-tP_0}\|_{\mL^1}\le C_K\,t^{-\frac n2},\qquad 0<t\le 1.
\end{equation}
This holds, for example, if the sectional curvatures of $(M,\g)$ are globally bounded above and below, and the injection radius is globally bounded below; see Lemma \ref{lem:traceclass} below.

If $M$ is compact, then $e^{-tP_0}$ is itself of trace class, and by the theorem of Minakshisundaram-Pleijel \cite{MP}, its trace admits a full asymptotic expansion as $t\rightarrow 0$, with trace coefficients expressed in terms of geometric invariants. For modern treatments of this result, see \cite{chow} and \cite{rosenberg}. Thus, for $M$ compact Theorem \ref{thm:main} states that $\trace(e^{-tP_V})$ admits an expansion
$$
\trace\bigl(e^{-tP_V}\bigr)=(4\pi t)^{-\frac n2}
\Bigl(c_0+c_1 t+c_2 t^2+\cdots+c_{m+1}t^{m+1}+\bigO\bigl(t^{m+2}\bigr)\Bigr),\quad 0< t\le 1,
$$
precisely when $V\in H^m(\R^n)$. Throughout this paper we are interested only in the trace near $t=0$, and henceforth in all statements we restrict to $t\in(0,1]$.

Theorem \ref{thm:main} is closely related to a priori estimates that give bounds on the Sobolev norms of a real, smooth potential $V$ in terms of the coefficients $c_k$. These bounds have been used to establish compactness in the $C^\infty$ topology of isospectral families of smooth potentials on a compact Riemannian manifold, with some a priori bound assumed on $V$ for dimensions $n\ge 4$. See for example 
\cite{MM},  \cite{Br}, and \cite{Don}. The novelty of Theorem \ref{thm:main} is to establish the regularity result analogous to these a priori bounds, for all finite orders of regularity. This requires in particular a careful analysis of the remainder terms in the heat trace expansion, for $t$ in an interval and $V$ of finite regularity, and not just of the coefficients $c_k$.

As an application of Theorem \ref{thm:main} we prove here the following result on existence of resonances for compact metric perturbations of the Laplacian. We remark that there exist complex valued $V$ with no resonances by \cite{TC}, and that even when $V\in C_{\comp}^\infty$ the result is known only in dimension three.
\begin{theorem}
Suppose that $M=\R^3$, and that $\g^{ij}(x)=\delta^{ij}$ on the complement of some compact set. Suppose also that $V\in L^\infty_{\comp}(\R^3)$ is real valued. Then the operator $P_V=-\Delta_\g+V$ has infinitely many scattering resonances, unless $V=0$ and $(M,\g)$ is isometric to Euclidean space.
\end{theorem}
\begin{proof}
This was proved in \cite{SaBZw} in the case $V\in C_{\comp}^\infty(\R^3)$, and in \cite{SmZw} for $V\in L^\infty_{\comp}(\R^3)$ in the case $\g^{ij}(x)=\delta^{ij}$ on all of $\R^3$. We follow here the proof in \cite{SmZw}, with the addition of a result from \cite{SaBZw}. To start, assume that there are no resonances. Then the argument in \cite[\S 2.3]{SmZw} shows that, since the scattering matrix is an entire function, the left hand side of \eqref{eqn:asymptotics} admits an asymptotic expansion with only terms of negative half-integral order. In particular, on $\R^3$ we have
$\trace\bigl(e^{-tP_V}-e^{-tP_0}\bigr)=c_0\,t^{-\frac 12}.$
By Theorem \ref{thm:main}, this implies that $V\in C_{\comp}^\infty(\R^3)$. We may then apply the Theorem of \cite{SaBZw} to see that $V=0$, and $(\R^3,\g)$ is isometric to Euclidean space.

We thus assume there is at least one resonance. Note, by Theorem \ref{thm:main} with $m=0$, that for some $c_0$
$$
\trace\bigl(e^{-tP_V}-e^{-tP_0}\bigr)=c_0 \, t^{-\frac 12}+\bigO(t^{\frac 12}).
$$
If there were only finitely many resonances, then the argument of \cite[\S 2.3]{SmZw} shows that
$$
\lim_{t\rightarrow 0^+}\Bigl(\trace\bigl(e^{-tP_V}-e^{-tP_0}\bigr)-c_0\, t^{-\frac 12}\Bigr)\ne 0,
$$
and hence there must in fact be infinitely many resonances.
\end{proof}

We conclude this section with three results concerning the heat kernel $e^{-tP_0}$ that will be used to obtain trace class bounds on $e^{-tP_V}-e^{-tP_0}$. A corollary of Lemma \ref{lem:traceclass} is that
\eqref{assump2} holds if the sectional curvatures are globally bounded above and below and there is a global lower bound on the injectivity radius. The first condition of the lemma holds in that case by \cite{cheng}, and the second by Bishop's volume comparison theorem \cite{Bishop}.

\begin{lemma}\label{lem:traceclass} 
Condition \eqref{assump2} holds if there is a constant $C$, and $x_0\in M$, such that when $t\in (0,1]$ and $R>0$,
$$
H_0(t,x,y)\le C\,t^{-\frac n2}e^{-\frac{d(x,y)^2}{Ct}}\,,\qquad
\mu(\{x:d(x,x_0)<2R\})\le C e^{CR^2},
$$
where $\mu$ is the Riemannian volume form for $\g$.
\end{lemma}
\begin{proof}
Let $w(x)=Cd(x,x_0)^2$, and write 
$$
\one_K e^{-tP_0}=\Bigl(\one_K e^{-\hf tP_0} e^{w}\Bigr)\Bigl(e^{-w}e^{-\hf tP_0}\Bigr)
$$
The second factor has Hilbert-Schmidt norm given by the square root of
$$
\int e^{-2w(x)}H_0(\tfrac 12 t,x,y)^2\,d\mu(y)\,d\mu(x)=\int e^{-2w(x)}H_0(t,x,x)\,d\mu(x).
$$
This in turn is bounded by
$$
C t^{-\frac n2}\int e^{-2Cd(x,x_0)^2}\,d\mu(x)\le Ct^{-\frac n2},
$$
where the last inequality follows easily from the bound on $\mu(B(x_0,R))$.
The first factor has Hilbert-Schmidt norm equal to the square root of
$$
\int\one_K(x)\,H_0(\tfrac 12 t,x,y)^2 e^{2Cd(y,x_0)^2}\,d\mu(y)\,d\mu(x).
$$
For $t<\frac 12C^{-2}$, we use the triangle inequality to dominate this by
\begin{multline*}
Ct^{-\frac n2}\int\one_K(x)\,e^{-8Cd(x,y)^2} e^{2Cd(y,x_0)^2}\,d\mu(y)\,d\mu(x)
\\
\le
Ct^{-\frac n2}
\biggl(\int\one_K(x)\,e^{8Cd(x,x_0)^2}\,d\mu(x)\biggr)
\biggl(\int e^{-2Cd(y,x_0)^2}\,d\mu(y)\biggr)\\
\le C_Kt^{-\frac n2},
\end{multline*}
and together these imply \eqref{assump2} for sufficiently small $t$. For $\frac 12 C^{-2}\le t\le 1$, \eqref{assump2} follows by the group property of the heat operator since $\mL^1$ is an ideal.
\end{proof}

In dimension $n\le 3$, the following will suffice to obtain the needed trace-norm estimates. Here, $\|\cdot\|_{\mL^2}$ is the Hilbert-Schmidt norm on operators.

\begin{lemma}\label{lem:hilbertschmidt}
Assume that $(M,\g)$ is complete, with global lower bounds on the Ricci curvature.
If $K$ is compact, then $\|\one_K\circ e^{-tP_0}\|_{\mL^2}\le C_K\,t^{-\frac n4}$.
\end{lemma}
\begin{proof}
We calculate the Hilbert-Schmidt norm of the kernel $H_0$ of $e^{-tP_0}$ with one variable restricted to $K$. Since $H_0>0$,
$$
\|\one_K\circ e^{-tP_0}\|_{\mL^2}^2=\int\one_K(x)H_0(t,x,y)^2\,d\mu(y)\,d\mu(x)=
\int_K H_0(2t,x,x)\,d\mu(x).
$$
This bounded by $\text{vol}(K)\,\sup_{x\in K}H_0(2t,x,x)$, and the result is a consequence of the following estimate, valid for compact subsets $K\subset M$,
\begin{equation*}
\sup_{x\in K}H_0(t,x,x)\le C_K\,t^{-\frac n2}.
\end{equation*}
This is known to hold if $M$ is compact, hence for $M$ as in the statement, by the following lemma.
\end{proof}

\begin{lemma}\label{diagest}
Suppose that $(\tilde M,\tilde\g)$ is a compact Riemannian manifold that isometrically contains a neighborhood $(U,\g)$ of the compact subset $K\subset M$. 
Let $\tilde H_0$ be the heat kernel on $(\tilde M,\tilde\g)$, and suppose $\chi\in C_{\comp}^\infty(U)$ equals $1$ on a neighborhood of $K$. Then, if $(M,\g)$ is complete with global lower bounds on its Ricci curvature, the following holds
$$
\sup_{x\in M, y\in K}\bigl|H_0(t,x,y)-\chi(x)\tilde H_0(t,x,y)\bigr|\le C_N\,t^N\quad\forall N,\;\;t\in(0,1].
$$
\end{lemma}
\begin{proof}
For $y\in K$, we consider $\chi(x)\tilde H_0(t,x,y)$ as a function of $x\in U\subset M$. Then by the local heat kernel expansion (see e.g.~\cite[(23.64)-(23.65)]{chow})
$$
\sup_{x\in U,y\in K}\bigl|(\partial_t-\Delta_\g)\chi(x)\tilde H_0(t,x,y)\bigr|\le C_N\,t^N\;\;\forall N,\quad\text{if}\quad t\in(0,1].
$$
Uniqueness of the heat kernel on $(M,\g)$ lets us write
\begin{multline*}
H_0(t,x,y)-\chi(x)\tilde H_0(t,x,y)\\=
\int_0^t \int_M H_0(t-s,x,z)(\partial_s-\Delta_\g)\bigl((\chi(z)\tilde H_0(s,z,y)\bigr)\,d\mu(z)\,ds.
\end{multline*}
Since the heat semigroup has norm 1 on $L^\infty(M)$, the right hand side vanishes to infinite order at $t=0$, uniformly over $y\in K$ and $x\in M$, leading to the desired estimate.
\end{proof}

The outline of this paper is as follows. In \S\ref{sec:heatkernel} we express $e^{-tP_V}$ as an iterative expansion, and use this, together with trace bounds for the localized heat kernel, to obtain an expansion for the trace of their difference. A key simplification is Corollary \ref{cor:reduc} where we reduce matters to the case of $M$ compact by compact support of $V$.
In \S\ref{sec:reductions}, we reduce the proof of Theorem \ref{thm:main} to Theorems \ref{thm:traceW2} and \ref{thm:traceWk}, the proofs of which are given in 
\S\ref{sec:traceW2} and \S\ref{sec:traceWk}. A key tool in these sections is the small time expansion for the heat kernel of $\Delta_\g$ near the diagonal, summarized in \S\ref{sec:heatkernel}, and the resulting rule \eqref{kernelprod} for the product of heat kernels.


\section{The Schr\"odinger heat kernel}\label{sec:heatkernel}
The proof of Theorem \ref{thm:main} relies on the following expansion for the Schr\"odinger heat kernel,
$$
e^{-tP_V}=e^{-tP_0}\,+\,\sum_{k=1}^\infty W_k(t),
$$
where
\begin{multline*}
W_k(t)=(-1)^k\int_{0<s_1<\cdots\,<s_k<t}e^{-(t-s_k)P_0}\,V\,e^{-(s_k-s_{k-1})P_0}\,V\,\cdots\\
\times V\,e^{-(s_2-s_1)P_0}\,V\,e^{-s_1P_0}ds_1\cdots ds_k.
\end{multline*}
The sum over $k$ converges for $t>0$ in the operator norm topology on $L^2(M)$, which follows since $\|W_k\|_{L^2\rightarrow L^2}\le t^k\|V\|_{L^\infty}^k/k!$. The latter bound follows since $\|e^{-tP_0}\|_{L^2\rightarrow L^2}\le 1$ for all $t\ge 0$, and since the region of integration over the $s$ variables has measure $t^k/k!$.

We now estimate the trace norm of the operators $W_k(t)$. Let $\mL^2$ denote the Hilbert-Schmidt operators on $L^2(M)$, and $\mL^1$ the trace class operators.
Recall that $\|ST\|_{\mL^1}\le \|S\|_{\mL^2}\|T\|_{\mL^2}$, and $\|ST\|_{\mL^1}\le \|S\|_{\mL^1}\|T\|_{L^2\rightarrow L^2}$.

Consider first the term $W_1(t)$, and $n\le 3$.  By Lemma \ref{lem:hilbertschmidt} with $K=\supp(V)$, we have
$$
\|e^{-(t-s)P_0}Ve^{-sP_0}\|_{\mL^1}\le C_{\supp(V)}^2\,\|V\|_{L^\infty}\,(t-s)^{-\frac n4}s^{-\frac n4}.
$$
For $n\le 3$ we can integrate this bound over $0<s<t$ to obtain
$$
\| W_1(t)\|_{\mL^1}\le C_{\supp(V)}^2\,\|V\|_{L^\infty}\,t^{1-\frac n2},\qquad n\le 3.
$$
Next consider $k\ge 2$, for $n\le 3$. If $s_j-s_{j-1}>(2k)^{-1}t$ for some $k\ge j\ge 1$, then
Lemma \ref{lem:hilbertschmidt} and the group property yield the bound 
\begin{multline*}
\|Ve^{-(s_j-s_{j-1})P_0}V\|_{\mL^1}\le C_{\supp(V)}\,2^{\frac n2}\,(s_j-s_{j-1})^{-\frac n2}\|V\|_{L^\infty}^2\\
\le C_{\supp(V)}(4k)^{\frac n2}\,t^{-\frac n2}\|V\|_{L^\infty}^2.
\end{multline*}
Since the volume of integration over the $s$ variables is $t^k/k!$, the integral over the region where $\max_j|s_j-s_{j-1}|>(2k)^{-1}t$  contributes to $W_k(t)$ a term with $\mL^1$ norm at most $C_{\supp(V)}(4k)^{\frac n2}(k!)^{-1}t^{k-\frac n2}\|V\|_{L^\infty}^k$.

If each $s_j-s_{j-1}\le (2k)^{-1}t$, then $t-s_k+s_1>\frac 12 t$, hence either $t-s_k>\frac 14t$ or $s_1>\frac 14t$. In the former case, using Lemma \ref{lem:hilbertschmidt}, $L^2$ boundedness of the heat kernel, and that $\mL^2$ is an ideal, we get the bound
\begin{multline*}
\|e^{-(t-s_k)P_0}\,V\,e^{-(s_k-s_{k-1})P_0}\,V\,\cdots V\,e^{-(s_2-s_1)P_0}\,V\,e^{-s_1P_0}\|_{\mL^1}\\
\le C_{\supp(V)}\|V\|_{L^\infty}^k\,t^{-\frac n4}\,s_1^{-\frac n4}.
\end{multline*}
Since the total volume of integration over this range of $(s_k,\ldots,s_2)$ is less than $(2k)^{-(k-1)}t^{k-1}$,
the contribution to $W_k(t)$ over this region has $\mL^1$ norm bounded by
$$
C_{\supp(V)}\|V\|_{L^\infty}^k\,(2k)^{-(k-1)}t^{k-\frac n2},\qquad n\le 3.
$$
Similar consideration of $s_1>\frac 14t$ leads to the same bound, and putting the above bounds together, using $k^k\ge k!$, we get
$$
\|W_k(t)\|_{\mL^1}\le C_{\supp(V)}\frac{k^{\frac n2}}{k!}\|V\|_{L^\infty}^k\, t^{k-\frac n2} ,\qquad n\le 3.
$$

The above argument fails if $n\ge 4$, so we assume \eqref{assump2} holds if $n\ge 4$.
Consider $W_k(t)$. For each $(s_k,\ldots,s_1)$ in the region of integration, at least one of the terms $t-s_k$ or $s_j-s_{j-1}$ is greater than $t/(k+1)$.
Applying assumption \eqref{assump2}, where $K=\supp(V)$, and the fact that $\mL^1$ is an ideal, we conclude
\begin{multline*}
\|e^{-(t-s_k)P_0}\,V\,e^{-(s_k-s_{k-1})P_0}\,V\,\cdots V\,e^{-(s_2-s_1)P_0}\,V\,e^{-s_1P_0}\|_{\mL^1}\\
\le C_{\supp(V)}\|V\|_{L^\infty}^k\,\Bigl(\frac t{k+1}\Bigr)^{-\frac n2}
\end{multline*}
uniformly over the region of integration. Thus, we again get the bound
$$
\|W_k(t)\|_{\mL^1}\le C_{\supp(V)}\frac{k^{\frac n2}}{k!}\|V\|_{L^\infty}^k\, t^{k-\frac n2}.
$$

In each case, the sum $\sum_{k=1}^\infty W_k(t)$ converges in $\mL^1$ for $t>0$, and bringing the trace into the sum we can write
$$
\trace\bigl(e^{-tP_V}-e^{-tP_0}\bigr)=\sum_{k=1}^\infty\trace\bigl(W_k(t)\bigr).
$$
Furthermore, by the above bounds on $\|W_k(t)\|_{\mL^1}$,
$$
\Bigl|\trace\bigl(e^{-tP_V}-e^{-tP_0}\bigr)-\sum_{k=1}^m\trace\bigl(W_k(t)\bigr)\Bigr|\le C_{V,m}\,t^{m+1-\frac n2}\,,\quad t\in(0,1].
$$

For $V\in L^\infty_{\comp}(M)$ the term $\trace\bigl(W_1(t)\bigr)$ has an expansion to all orders in $t$. To see this, write
\begin{align*}
\trace\bigl(W_1(t)\bigr)&\,=\,-\int_0^t \int_{M\times M}H_0(t-s,x,y)\,V(y)\,H_0(s,y,x)\,d\mu(x)\,d\mu(y)\,ds \\
&\,=\,-t\int_M H_0(t,y,y)\,V(y)\,d\mu(y),\rule{0pt}{18pt}
\end{align*}
where we use the group property of the heat kernel. The expansion of $H_0(t,y,y)$, see \eqref{kernelform} below, yields an expansion for $\trace\bigl(W_1(t)\bigr)$ of the form on the right hand side of \eqref{eqn:asymptotics} for arbitrary positive integer $m$.

Generally, for all $k$ the function $\trace\bigl(W_k(t)\bigr)$ involves $H_0(t,x,y)$ only for $x,y\in\supp(V)$. This follows since, after using the composition rule, we can write 
$(-1)^k\,\trace\bigl(W_k(t)\bigr)$, for $k\ge 2$ and $t>0$, as
\begin{multline*}
\int_{0<s_1<\cdots\,<s_k<t}\int_{M^k}H_0(t+s_1-s_k,y_1,y_k)\,H_0(s_k-s_{k-1},y_k,y_{k-1})\cdots \\
\times H_0(s_2-s_1,y_2,y_1) V(y_k)\,\cdots\,V(y_1)\,d\mu(y_1)\cdots d\mu(y_k)\,ds_1\cdots ds_k.
\end{multline*}

Let $\Lambda^{k-1}\subset\R^k$ be the $(k-1)$-simplex, consisting of $\bfr=(r_1,\ldots,r_k)$ with $r_j>0$ for all $j$, and with $r_1+\cdots+r_k=1$. Let $d\bfr$ be the measure on $\Lambda^{k-1}$ induced by projection onto $(r_2,\ldots,r_k)$, and let $\bfy=(y_1,\ldots,y_k)\in M^k$. Then, by cyclicity of the integrand, we can write $(-1)^k\,\trace\bigl(W_k(t)\bigr)$ as
\begin{multline}\label{traceform}
\frac{t^k}{k}\int_{\Lambda^{k-1}}\int_{M^k}H_0(tr_k,y_1,y_k)\,H_0(tr_{k-1},y_k,y_{k-1})\cdots H_0(tr_1,y_2,y_1)\\
\times V(y_k)\,\cdots\,V(y_1)\,d\mu(\bfy)\,d\bfr.
\end{multline}
We then have the following simple corollary of Lemma \ref{diagest}, which allows us to henceforth reduce matters to the case of $M$ compact.
\begin{corollary}\label{cor:reduc}
Suppose that $(\tilde M,\tilde g)$ is a compact Riemannian manifold that isometrically contains a neighborhood of $\supp(V)\subset M$. Then Theorem \ref{thm:main} holds for $V$ on $(M,\g)$ iff it holds for $V$ on $(\tilde M,\tilde g)$.
\end{corollary}

We now recall the construction of an asymptotic formula for $H_0(t,x,y)$ on compact $M$, for example as in \cite[Chapter 23]{chow} and \cite[\S 3.2]{rosenberg}. Choose $c\le 1$ such that the injectivity radius at each point in $M$ is greater than $c$. Define $E(t,x,y)$ on $M\times M$ by
$$
E(t,x,y)=(4\pi t)^{-\frac n2}e^{-\frac{d(x,y)^2}{4t}}.
$$
This is a smooth function on $\{(0,\infty)\times M^2\,:\,d(x,y)<c\}$.
Then there are real valued 
$u_k(x,y)\in C^\infty(M^2)$, supported in $\{d(x,y)<c\}$, such that
\begin{equation}\label{kernelform}
H_0(t,x,y)=E(t,x,y)\sum_{k=0}^N t^k\,u_k(x,y)\,+\,w_N(t,x,y),
\end{equation}
where
$$
u_k(x,y)=u_k(y,x),\qquad u_0(y,y)=1,
$$
and
\begin{equation}\label{wNform}
\bigl|w_N(t,x,y)\bigr|\le C_{N}\,t^{N+1-\frac n2}e^{-\frac{d(x,y)^2}{8t}}, \quad t\in(0,1].
\end{equation}
Note that $w_N\in C^\infty\bigl((0,\infty)\times M^2\bigr)$
since the other terms in \eqref{kernelform} are.

As a corollary of \eqref{kernelform}--\eqref{wNform}, we have the following multiplicative relation, where $0<v<1$,
\begin{equation}\label{kernelprod}
H_0(vt,x,y)\, H_0((1-v)t,x,y)=(4\pi t)^{-\frac n2}\Bigl(H_0(v(1-v)t,x,y)+R(t,v,x,y)\Bigr),
\end{equation}
where $R(t,v,x,y)\in C^\infty\bigl((0,\infty)\times(0,1)\times M^2\bigr)$, and for each $N$,
\begin{equation}\label{Rform}
R(t,v,x,y)=E(v(1-v)t,x,y)\sum_{k=0}^N\sum_{j=0}^{2k} t^k\,v^j\,r_{k,j}(x,y)\,+\,R_N(t,v,x,y),
\end{equation}
where
\begin{equation}\label{rkform}
r_{k,j}(x,y)=r_{k,j}(y,x),\qquad r_{0,0}(x,y)=u_0(x,y)^2-u_0(x,y),
\end{equation}
and
\begin{equation}\label{RNform}
|R_N(t,v,x,y)|
\le C_{N}\,t^{N+1}E\bigl(2v(1-v)t,x,y\bigr),\quad t\in(0,1].
\end{equation}


\section{Preliminary reductions}\label{sec:reductions}

By the results of \S \ref{sec:heatkernel}, it suffices to establish the analogue of Theorem \ref{thm:main} where $\trace\bigl(e^{-tP_V}-e^{-tP_0}\bigr)$ is replaced by 
$\sum_{k=2}^\infty \trace\bigl(W_k(t)\bigr)$, on a compact Riemannian manifold $(M,\g)$. 
In this section we reduce matters to the following two theorems.

\begin{theorem}\label{thm:traceW2}
If $V\in L^\infty_{\comp}\cap H^m(M)$ is real valued, then one can write
\begin{equation*}
\trace\bigl(W_2(t)\bigr)=(4\pi t)^{-\frac n2}
\Bigl(c_{2,2} t^2+\cdots+c_{2,2+m}t^{2+m}+\epsilon(t)t^{2+m}\Bigr),
\end{equation*}
where $\lim_{t\rightarrow 0^+}\epsilon(t)=0$.

Conversely, assume $V\in L^\infty_{\comp}\cap H^{m-1}(M)$ is real valued. If one can write
\begin{equation}\label{eqn:traceW2'}
\trace\bigl(W_2(t)\bigr)=(4\pi t)^{-\frac n2}
\Bigl(c_{2,2} t^2+\cdots+c_{2,1+m}t^{1+m}+r_{2,2+m}(t)t^{2+m}\Bigr),
\end{equation}
where $|r_{2,2+m}(t)|\le C$ for $0<t\le 1$, then necessarily $V\in H^{m}(M)$.
\end{theorem}

\begin{theorem}\label{thm:traceWk}
If $V\in L^\infty_{\comp}\cap H^m(M)$ is real valued, then for $k\ge 2$ one can write
$$
\trace\bigl(W_k(t)\bigr)=(4\pi t)^{-\frac n2}
\Bigl(c_{k,k} t^k+\cdots+c_{k,k+m-1}t^{k+m-1}+r_{k,k+m}(t)t^{k+m}\Bigr),
$$
where for $0\le j\le m$, and a constant $C$ depending on $k$ and $m$, 
$$
|c_{k,k+j}|\le C\,\|V\|_{L^\infty}^{k-2}\|V\|_{H^j}^2,\qquad
\sup_{0<t\le 1}|r_{k,k+m}(t)|\le C\,\|V\|_{L^\infty}^{k-2}\|V\|_{H^m}^2.
$$
\end{theorem}

That $V\in L^\infty_{\comp}\cap H^m(M)$ implies existence of the asymptotic expansion \eqref{eqn:asymptotics} of order $m+2$ follows easily from these two theorems: by the bound $\|W_k(t)\|_{\mL^1}\le C^k\,k^{\frac n2}\,t^{k-\frac n2}/k!$, we see that
\begin{equation}\label{eqn:traceerror}
\trace\sum_{k=m+3}^\infty W_k(t)\le C\,t^{m+3-\frac n2},\qquad 0<t\le 1.
\end{equation}
On the other hand, Theorems \ref{thm:traceW2} and \ref{thm:traceWk} show that, with $c_j=\sum_{k=2}^{j}c_{k,j},$
$$
\trace\sum_{k=2}^{m+2} W_k(t)=(4\pi t)^{-\frac n2}
\Bigl(c_2 t^2+\cdots+c_{m+1}t^{m+1}+c_{m+2}t^{m+2}+\epsilon(t)t^{m+2}\Bigr).
$$
 
The other direction of Theorem \ref{thm:main}, that existence of an asymptotic expansion implies regularity, is carried out by induction. 
Assume $m\ge 1$ and $V\in L^\infty_{\comp}\cap H^{m-1}(M)$, which trivially holds if $m=1$ since $L^\infty_{\comp}(M)\subset L^2(M)$, and assume
\eqref{eqn:asymptotics} holds. By \eqref{eqn:traceerror} this implies that, with $|r_{m+2}(t)|\le C$,
$$
\trace\sum_{k=2}^{m+2} W_k(t)=(4\pi t)^{-\frac n2}
\Bigl(c_2 t^2+\cdots+c_{m+1}t^{m+1}+r_{m+2}(t)t^{m+2}\Bigr).
$$
Since $V\in L^\infty_{\comp}\cap H^{m-1}(M)$, Theorem \ref{thm:traceWk} shows that $\trace\sum_{k=3}^{m+2}W_k(t)$ has a similar expansion, with coefficients that are bounded in terms of the $L^\infty$ and $H^j$ norms of $V$ with $j\le m-1$. Hence
the relation \eqref{eqn:traceW2'} holds, and we conclude $V\in H^m(M)$.

In proving Theorems \ref{thm:traceW2} and \ref{thm:traceWk}, we will use a simple calculus lemma.

\begin{lemma}\label{lem:expansion}
Suppose that $f\in C^\infty\bigl((0,1)\bigr)$, and that for all $0\le j\le m$
$$
\sup_{0<t<1}|f^{(j)}(t)|\le C_j
$$
for finite constants $C_j$. Then for $t\in (0,1)$ one has
$$
f(t)=a_0+a_1 t+\cdots +a_{m-1}t^{m-1}+r_m(t)\,t^m,
$$
where $\sup_{0<t<1}|r_m(t)|<C_m/m!,$ and $|a_j|\le C_j/j!.$
\end{lemma}
The lemma is proved taking the Taylor expansion about $\epsilon>0$, then letting $\epsilon\rightarrow 0^+$, using that $\lim_{\epsilon\rightarrow 0^+}f^{(j)}(\epsilon)$ exists if $0\le j\le m-1$.


\section{Proof of Theorem \ref{thm:traceW2}}\label{sec:traceW2}
In this, and subsequent sections, we will assume $M$ is compact.
We reduce the proof of Theorem \ref{thm:traceW2} to that of two propositions, which we then prove in this section.
As in \S\ref{sec:heatkernel}, we write $\trace\bigl(W_2(t)\bigr)$ as
$$
\tfrac 12 \, t^2\int_0^1\int_{M^2}
H_0((1-v)t,y,z)H_0(vt,y,z)\,V(y)\,V(z)
\,d\mu(y)\,d\mu(z)\,dv.
$$
We now apply the relation \eqref{kernelprod}. We show that the remainder $R$ leads to a term that is better by one power of $t$ than the main term, for $V$ of a given Sobolev regularity.

\begin{proposition}\label{prop:traceR}
If $V\in L^\infty_{\comp}\cap H^{m-1}(M)$, then one can write
\begin{multline}\label{eqn:traceR}
\int_0^1\int_{M^2}
R(t,v,y,z)\,V(y)\,V(z)
\,d\mu(y)\,d\mu(z)\,dv\\
=a_1 t+\cdots+a_{m-1}t^{m-1}+r_m(t)t^m,
\end{multline}
where, for fixed constants $C_j$, and $1\le j\le m-1$,
$$
|a_j|\le C_j\,\|V\|_{H^{j-1}}^2,\qquad
\sup_{0<t<1}|r_m(t)|\le C_m\,\|V\|_{H^{m-1}}^2.
$$
\end{proposition}
A simple induction argument shows that Theorem \ref{thm:traceW2} is a consequence of Proposition \ref{prop:traceR} and the following.

\begin{proposition}\label{prop:traceH2}
If $V\in L^\infty_{\comp}\cap H^m(M)$ is real valued, then one can write
\begin{multline}\label{eqn:traceH2}
\int_0^1\int_{M^2}
H_0(v(1-v)t,y,z)\,V(y)\,V(z)
\,d\mu(y)\,d\mu(z)\,dv\\
=a_0 +\cdots+a_mt^m+\epsilon(t)t^m,
\end{multline}
where $\lim_{t\rightarrow 0^+}\epsilon(t)=0.$

Conversely, assuming $V\in L^\infty_{\comp}\cap H^{m-1}(M)$ is real valued, if one has
\begin{multline}\label{eqn:traceH2'}
\int_0^1\int_{M^2}
H_0(v(1-v)t,y,z)\,V(y)\,V(z)
\,d\mu(y)\,d\mu(z)\,dv\\
=a_0 +\cdots+a_{m-1}t^{m-1}+r_m(t)t^m,
\end{multline}
where $|r_m(t)|\le C_m$ for $0<t\le 1$, then $V\in H^{m}(M)$, and hence \eqref{eqn:traceH2} holds.
\end{proposition}


\subsection{Proof of Proposition \ref{prop:traceH2}}

Since $M$ is compact, we can expand $V$ in a basis of eigenfunctions $V=\sum_{j=1}^\infty b_j\phi_j$, where $-\Delta_\g \phi_j=\rho_j\phi_j$, and $\rho_j\ge 0$. Since $V$ is real-valued, the left hand side of \eqref{eqn:traceH2} equals
$$
\int_0^1\sum_{j=1}^\infty e^{-v(1-v)t\rho_j}\,|b_j|^2\,dv.
$$
We use the equivalence
$$
\|V\|_{H^m}^2\approx \sum_{j=1}^\infty(1+\rho_j^{m})\,|b_j|^2.
$$

Consider $m=1$, and suppose that the expansion \eqref{eqn:traceH2'} holds, hence that
$$
\int_0^1\sum_{j=1}^\infty e^{-v(1-v)t\rho_j}\,|b_j|^2\,dv=a_0+r_1(t)t.
$$
Letting $t\rightarrow 0$ gives $a_0=\sum|b_j|^2=\|V\|_{L^2}^2<\infty$, so we can rewrite this as
$$
\int_0^1\sum_{j=1}^\infty \biggl(\frac{1-e^{-v(1-v)t\rho_j}}{t}\biggr)\,|b_j|^2\,dv \le |r_1(t)|, \quad t\in(0,1].
$$
The integrand is positive, so applying Fatou's lemma as $t\rightarrow 0$ we get
$$
\biggl(\int_0^1 v(1-v)\,dv\biggr)\sum_{j=1}^\infty\rho_j\,|b_j|^2  \le C_1,
$$
implying that $V\in H^1(M)$. Conversely, if $V\in H^1(M)$ we would get an expansion of the form \eqref{eqn:traceH2} with $m=1$ by dominated convergence.

To consider higher values of $m$, write
$$
e^{-s}\;=\;\sum_{j=0}^{m-1}\frac{(-1)^j}{j!}\,s^j\,+\,e_m(s)\,\frac{(-1)^m}{m!}\,s^m.
$$
Then
\begin{equation}\label{errorterm}
0\le e_m(s)\le 1\;\;\;\text{if}\;\;\;s\ge 0,\qquad
\lim_{s\rightarrow 0}\,e_m(s)=1.
\end{equation}
The proof of \eqref{eqn:traceH2} for $V\in L^\infty\cap H^m(M)$ follows by dominated convergence.

Suppose then that $V\in L^\infty\cap H^{m-1}(M)$ for some $m\ge 1$, and that \eqref{eqn:traceH2'} holds. By induction, or comparison with \eqref{eqn:traceH2}, we must have
$$
a_k=\biggl(\frac {(-1)^k}{k!}\int_0^1 v^k(1-v)^k\,dv\biggr)
\sum_{j=1}^\infty \rho_j^{k}\,|b_j|^2,\qquad 0\le k\le m-1.
$$
We can then expand
\begin{multline*}
\int_0^1 \sum_{j=1}^\infty e^{-v(1-v)t\rho_j}\,|b_j|^2\,dv=\\
\sum_{k=0}^{m-1}a_k\,t^k
+\frac{(-1)^m}{m!}\biggl(\int_0^1\sum_{j=1}^\infty e_m\bigl(v(1-v)t\rho_j\bigr)\,v^m(1-v)^m\,\rho_j^{m}\,|b_j|^2\,dv\biggr)\,t^m.
\end{multline*}
We thus must have uniform bounds for $t\in(0,1]$
$$
\int_0^1\sum_{j=1}^\infty e_m\bigl(v(1-v)t\rho_j\bigr)\,v^m(1-v)^m\,\rho_j^{m}\,|b_j|^2\,dv \le m!\ts C_m.
$$
By Fatou's lemma and \eqref{errorterm}, we deduce that
$$
\biggl(\int_0^1 v^m(1-v)^m\,dv\biggr)\,
\sum_{j=1}^\infty\,\rho_j^{m}\,|b_j|^2\le m!\ts C_m,
$$
so necessarily $V\in H^m(M)$, completing the proof of Proposition \ref{prop:traceH2}. 
\qed


\subsection{Proof of Proposition \ref{prop:traceR}}

We use the expansion \eqref{Rform}. By the Schur test and \eqref{RNform}, we can bound
$$
\biggl|\;\int_0^1\int_{M^2}
R_N(t,v,y,z)\,V(y)\,V(z)
\,d\mu(y)\,d\mu(z)\,dv\;\biggr|\le C_{N,\Omega}\,t^{N+1}\,\|V\|_{L^2}^2.
$$
Taking $N=m$, it suffices to establish the expansion in \eqref{eqn:traceR} for each of the other terms in \eqref{Rform}.
Other than the term $k=0$, this is handled by the following.

\begin{lemma}\label{lem:highterms}
Suppose that $V\in L^\infty\cap H^{m-1}(M)$, $M$ compact, and that $r(x,y)\in C^\infty(M^2)$ is supported in $d(x,y)<c$. Then, one can write
\begin{multline*}
\int_0^1\int_{M^2}
E(v(1-v)t,x,y)\,r(x,y)\,V(y)\,V(x)
\,d\mu(y)\,d\mu(x)\,dv\\
=a_0+a_1 t+\cdots+a_{m-2}t^{m-2}+r_{m-1}(t)t^{m-1},
\end{multline*}
where, for constants $C_j$,
$$
|a_j|\le C_j\,\|V\|_{H^j}^2,\qquad
\sup_{0<t<1}|r_{m-1}(t)|\le C_{m-1}\,\|V\|_{H^{m-1}}^2.
$$
\end{lemma}
\begin{proof}
For any $\delta>0$, the kernel $E(v(1-v)t,x,y)$ is smooth over $v\in[0,1]$ and $t\ge 0$ for $d(x,y)>\delta$, hence we can use a partition of unity to reduce to the case that $V$ is supported in a local coordinate neighborhood, over which we fix an orthornormal frame on $T(M)$.
Write $y=\exp_x(z)$, $z\in \R^n\equiv T_x(M)$ via the frame. We absorb the Jacobian factors $D\mu(y)/Dz$ and $D\mu(x)/Dx$ into the smooth function $r(x,z)$, supported in $|z|<c$, and consider
\begin{equation}\label{Rterm}
\int_0^1\int_{\R^{2n}}
\bigl(4\pi v(1-v)t\bigr)^{-\frac n2}\,e^{-\frac{|z|^2}{4v(1-v)t}}\,r(x,z)\,V(\exp_x(z))\,V(x)
\,dz\,dx\,dv.
\end{equation}

Applying $\partial_t$ to the integrand in \eqref{Rterm} is equivalent, after integrating by parts in $z$, to applying $v(1-v)\Delta_z$ to $r(x,z)V(\exp_x(z))$. Using the following lemma and integrating by parts in $x$, we can convert half of the $z$-derivatives falling on $V(\exp_x(z))$ into $x$-derivatives acting on either $r(x,z)$ or the other factor $V(x)$. 

\begin{lemma}\label{utox}
Given a local coordinate chart, and orthonormal frame on $T(M)$ over the chart,
then for $z\in\R^n$ with $|z|<c$, there are smooth first order differential operators $A_j(x,z,\partial_x)$ and $B_j(x,z,\partial_z)$, so that
\begin{align*}
\partial_{z_j}V(\exp_x(z))&=A_j(x,z,\partial_x)V(\exp_x(z)),\\
\partial_{x_j}V(\exp_x(z))&=B_j(x,z,\partial_z)V(\exp_x(z)).\rule{0pt}{15pt}
\end{align*}
\end{lemma}

This lets us express the $j$-th derivative with respect to $t$ of \eqref{Rterm} as a sum \begin{multline*}
\sum_{|\alpha|,|\beta|\le j}\int_0^1\int_{\R^{2n}}
\bigl(4\pi v(1-v)t\bigr)^{-\frac n2}\,e^{-\frac{|z|^2}{4v(1-v)t}}\,r_{\alpha,\beta}(v,x,z)\\
\times \partial_z^\alpha V(\exp_x(z))\,\partial_x^\beta V(x)
\,dz\,dx\,dv
\end{multline*}
where $r_{\alpha,\beta}(v,x,z)$ is a smooth function supported in $|z|<c$.
Changing variables back to $(x,y)\in M\times M$, we apply the
Schur test to $E(v(1-v)t,x,y)$ to bound this by $\|V\|_{H^j}^2$, with bounds uniform in $t$. The result now follows by Lemma \ref{lem:expansion}.
\end{proof}

To handle the remaining term $k=0$, we will use that
$$
\bigl|r_{0,0}(x,\exp_x(z))\bigr|\le C\,|z|^2.
$$
To see this, note that $r_{0,0}(x,x)=0$ and $r_{0,0}(x,y)=r_{0,0}(y,x)$ by \eqref{rkform}, which together imply that $\nabla_{y}r_{0,0}(x,y)|_{y=x}=0$. Taking a Taylor expansion of $r_{0,0}(x,\exp_x(z))$ about $z=0$ thus reduces matters to showing that, when $V\in L^\infty\cap H^{m-1}(M)$,
\begin{multline*}
\int_0^1\int_M\int_{\R^n}
\bigl(4\pi v(1-v)t\bigr)^{-\frac n2}\,\langle A(x)z,z\rangle\, e^{-\frac{|z|^2}{4v(1-v)t}}\\ \times r(x,z)\,V(\exp_x(z))\,V(x)
\,dz\,d\mu(x)\,dv\\
=a_1 t+\cdots+a_{m-1}t^{m-1}+r_m(t)t^m
\end{multline*}
with coefficients $a_j$ satisfying the bounds of Proposition \ref{prop:traceR}. Here $r(x,z)$ is assumed smooth and supported in $|z|<c$, and $A(x)$ is a smooth, symmetric matrix valued function of $x\in M$. We use the identity (see Lemma \ref{quadexpident} below)
$$
\langle A(x)z,z\rangle\, e^{-\frac{|z|^2}{4s}}=
s\,2\trace(A(x))\,e^{-\frac{|z|^2}{4s}}+4s^2\langle A(x)\partial_z,\partial_z\rangle\,e^{-\frac{|z|^2}{4s}}.
$$
With $s=v(1-v)t$, the first term on the right is handled by Lemma
\ref{lem:highterms}.  For the second term, we integrate by parts as above to distribute one derivative on each of the two factors of $V$, and apply Lemma \ref{lem:highterms} with $m-1$ replaced by $m-2$.
\qed


\section{Proof of Theorem \ref{thm:traceWk}}\label{sec:traceWk}
We will use the following version of the Gagliardo-Nirenberg inequalities. Recall that we assume $(M,\g)$ is a compact Riemannian manifold.
\begin{lemma}
Suppose that $m_j\le m$, and $\sum_{j=1}^k m_j=2m$. Then
$$
\bigl\|\prod_{j=1}^k \bigl|\nabla^{m_j}u_j\bigr|\,\bigr\|_{L^1}\le C\,
\Bigl(\;\sum_{j=1}^k \|u_j\|_{L^\infty}\Bigr)^{k-2}
\Bigl(\;\sum_{j=1}^k \|u_j\|_{H^m}\Bigr)^2.
$$
\end{lemma}
\begin{proof}
By a partition of unity we can work with smooth cutoffs of $u_j$ in local coordinates, with the standard gradient $\nabla$, and with the Sobolev space $H^m(\R^n)$.
We apply the following bound, see \cite[(3.17)]{pde}, 
where we assume $u\in L^\infty\cap H^m$,
$$
\|\nabla^{m_j}u_j\|_{L^{\frac {2m}{m_j}}}\le C\,
\|u_j\|_{L^\infty}^{1-\frac{m_j}m}\|\nabla^m u_j\|_{L^2}^{\frac{m_j}m}
$$
and use H\"older's inequality after taking the product over $j$.
\end{proof}

Recall the formula \eqref{traceform}.
For any $\delta>0$, the kernel $H_0(t,y,z)$ belongs to $C^\infty(\R_+\times M\times M)$ on the set $d(y,z)>\delta$, with uniform bounds over $t\in(0,1]$, and all derivatives vanish to infinite order at $t=0$. Hence, as in the proof of Lemma \ref{lem:highterms}, for a small $c>0$ to be chosen, we may restrict to the case that $V(y)$ is supported in the set $U=\{y:d(y,x_0)<c\}$ for some point $x_0$. 
From the expansion \eqref{kernelform}, it suffices to show that when $f(t)$ takes the form
\begin{multline*}
\int_{\!\Lambda^{k-1}\times M^k\rule{0pt}{16pt}}\!\!\!\!\!\!\!\!\!\!\!\!(4\pi t)^{-\frac{n(k-1)}2}\Bigl(\prod_{j=1}^k r_j^{-\frac n2}\Bigr)
e^{-\tfrac 1{4t}\bigl(r_k^{-1}d^2(y_1,y_k)+r_{k-1}^{-1}d^2(y_k,y_{k-1})+\,\cdots\, +r_1^{-1}d^2(y_2,y_1)\bigr)}\\
\times V_k(y_k)\,\cdots\,V_1(y_1)\,\phi(\bfy)\,d\mu(\bfy)\,d\bfr,
\end{multline*}
where $\bfy=(y_1,\ldots,y_k)$, and $\phi(\bfy)\in C_{\comp}^\infty(U^k)$, then if each $V_j\in L^\infty_{\comp}\cap H^m(M)$ we can write
$$
f(t)=a_0+a_1 t+\cdots +a_{m-1}t^{m-1}+r_m(t)\ts t^m,
$$
with bounds on $|a_j|$ and $\|r_m\|_{L^\infty}$ as in the statement of Theorem \ref{thm:traceWk}.

We fix local coordinates on $U$ to identify $y_1$ with $x\in \R^n$, and fix an orthonormal frame over $U$. For each $x\in U$, and $c$ sufficiently small, this induces exponential coordinates $e_x(u)\equiv\exp_x(u)$ on $U$, based at $x$.
We then set $y_j=e_x(u_j)$ for $2\le j\le k$, so that $(x,u_2,\ldots,u_k)$ are coordinates on the support of $\phi(\bfy)$.

After absorbing $d\mu(\bfy)/d\bfu\,dx$ into $\phi$, we express $f(t)$ as
\begin{multline*}
\int_{\Lambda^{k-1}}(4\pi t)^{-\frac{n(k-1)}2}\Bigl(\prod_{j=1}^k r_j^{-\frac n2}\Bigr)
\int_{\R^{nk}}
e^{-\tfrac 1{4t}D(x,\bfr,\bfu)}\\
\times V_k\bigl(e_x(u_k)\bigr)\,\cdots\,V_2\bigl(e_x(u_2)\bigr)\,V_1(x)\,\phi(x,\bfu)\,d\bfu\,dx\,d\bfr.
\end{multline*}
Here, $\supp(\phi)\subset \{|u_j|<2c\}$ for all $j$,
$\bfu=(u_2,\ldots,u_k)$, and
\begin{multline*}
D(x,\bfr,\bfu)=r_k^{-1}|u_k|^2+r_{k-1}^{-1}d^2(e_x(u_k),e_x(u_{k-1}))+\cdots\\
+r_2^{-1}d^2(e_x(u_3),e_x(u_2))+r_1^{-1}|u_2|^2.
\end{multline*}
For $u,v\in\R^n$ with $|u|,|v|<2c$, and some $C<\infty$,
$$
C^{-1}|u-v|^2\le d^2(e_x(u),e_x(v))\le C\,|u-v|^2.
$$
Consequently, by the analysis of \cite[(3.12)]{SmZw}, we have uniform bounds over $\bfr\in\Lambda^{k-1}$,
\begin{equation}\label{3.12}
(4\pi t)^{-\frac{n(k-1)}2}\Bigl(\prod_{j=1}^k r_j^{-\frac n2}\Bigr)
\int_{\R^{n(k-1)}}
\sup_{x\in M}\,e^{-\tfrac 1{4t}D(x,\bfr,\bfu)}\,d\bfu\le C.
\end{equation}
Hence,
\begin{equation}\label{fbound}
|f(t)|\le C\,\sup_{\bfu}\int |V_k\bigl(e_x(u_k)\bigr)\,\cdots\,V_2\bigl(e_x(u_2)\bigr)\,V_1(x)\,\phi(x,\bfu)|\,dx.
\end{equation}
Taking $c$ smaller if necessary, the map $x\rightarrow e_x(u)$ is a diffeomorphism for $|u|<2c$, and so by H\"older's inequality we have bounds
$$
|f(t)|\,\le\, C\,\prod_{j=1}^k \|V_j\|_{L^{p_j}}\quad \text{if}\quad
\sum_{j=1}^k \, p_j^{-1}=1.
$$
This establishes the case $m=0$ of Theorem \ref{thm:traceWk}, taking $p_1=p_2=2$, and $p_j=\infty$ for $j\ge 3$.

To consider derivatives of $f(t)$, we observe that the symmetric function $d^2(e_x(u),e_x(v))$ vanishes to second order at $u=v$, and hence
$$
d^2(e_x(u),e_x(v))=\sum_{i,j=1}^nq_{ij}(x,u,v)(u^i-v^i)(u^j-v^j),
$$
with $q_{ij}(x,u,v)$ symmetric in $ij$, and depending smoothly on $x,u,v\in U$. Furthermore, $q_{ij}(x,0,0)=\delta_{ij}$, since $d^2(e_x(u),e_x(0))=|u|^2.$ Taking the Taylor expansion of $q_{ij}(x,u,v)$ in $u$ and $v$ lets us write
\begin{multline}\label{distform}
d^2(e_x(u),e_x(v))=|u-v|^2\,+\!\!\sum_{1\le|\alpha+\beta|<N}u^\alpha v^\beta Q_{\alpha\beta,x}(u-v)\\
+\!\!\sum_{|\alpha+\beta|=N}u^\alpha v^\beta R_{\alpha\beta,x}(u,v),
\end{multline}
where $Q_{\alpha\beta,x}$ are quadratic forms in $u-v$ that depend smoothly on $x$, and $R_{\alpha\beta,x}(u,v)$ is smooth in $(x,u,v)$ and satisfies
$R_{\alpha\beta,x}(u,v)\le C_{\alpha\beta}|u-v|^2$.

For $\bfr\in\Lambda^{k-1}$, let $Q_\bfr(\bfu)$ denote the quadratic form on  $\R^{n(k-1)}$,
$$
Q_\bfr(\bfu)=r_k^{-1}|u_k|^2+r_{k-1}^{-1}|u_k-u_{k-1}|^2+\,\cdots\,+r_2^{-1}|u_3-u_2|^2
+r_1^{-1}|u_2|^2.
$$
Then, for all $\bfr\in\Lambda^{k-1}$ and $x$, $\bfu$ in the support of $\phi(x,\bfu)$, for $c$ sufficiently small,
$$
\tfrac 12 Q_\bfr(\bfu)\le D(x,\bfr,\bfu)\le 2\,Q_\bfr(\bfu).
$$

Also, by the above we can write:
\begin{equation}\label{qexpansion}
D(x,\bfr,\bfu)=Q_\bfr(\bfu)\,+\!\!
\sum_{1\le|\alpha|<N}\bfu^\alpha Q_{\alpha,\bfr,x}(\bfu)+\!\!
\sum_{|\alpha|=N}\bfu^\alpha R_{\alpha,\bfr,x}(\bfu)\rule{0pt}{15pt}
\end{equation}
where $Q_{\alpha,\bfr,x}(\bfu)$ are quadratic forms, the $R_{\alpha,\bfr,x}(\bfu)$ are smooth functions, and where, with constants $C_{\alpha,\beta}$ uniform over $\bfr\in\Lambda^{k-1}$ and $x$, $\bfu\in U$,
\begin{equation}\label{dominate}
|\partial_x^\beta Q_{\alpha,\bfr,x}(\bfu)|\le C_{\alpha,\beta}\,Q_\bfr(\bfu),\qquad\quad
|\partial_x^\beta R_{\alpha,\bfr,x}(\bfu)|\le C_{\alpha,\beta}\,Q_\bfr(\bfu).
\end{equation}
The key point to the bounds \eqref{dominate} is that, although the various quadratic forms have singular behavior in $\bfr$, for $1<j<k$ the terms $Q_{\alpha\beta,x}(u_{j+1}-u_j)$ and $R_{\alpha\beta,x}(u_{j+1}-u_j)$  in \eqref{distform}, which multiply against $r_j^{-1}$, are dominated, as are their derivatives in $x$, by the corresponding term $|u_{j+1}-u_j|^2$ in $Q_\bfr(\bfu)$.

We next note the bound, uniformly over $\bfr\in\Lambda^{k-1}$,
\begin{equation}\label{prodbound}
(4\pi t)^{-\frac{n(k-1)}2}\Bigl(\prod_{j=1}^k r_j^{-\frac n2}\Bigr)
\int_{\R^{n(k-1)}}
\bigl|\bfu^\alpha\bigr|\, Q_\bfr(\bfu)^j\, e^{-\tfrac 1{4t}Q_\bfr(\bfu)}\,d\bfu\le C_{j,\alpha}\,t^{j+\frac 12|\alpha|},
\end{equation}
which is a simple variation on \eqref{3.12}, and the fact that $Q_\bfr(\bfu)\ge c\,|\bfu|^2$ for $\bfr\in\Lambda^{k-1}$. Since the estimate \eqref{prodbound} involves only absolute bounds, it also holds when the term $Q_\bfr(\bfu)^j$ is replaced by a $j$-fold product of quadratic forms $Q_{\alpha,\bfr,x}(\bfu)$ from \eqref{qexpansion}.

Consequently, if we expand $\exp\bigl(-\bigl(D(x,\bfr,\bfu)-Q_\bfr(\bfu)\bigr)/4t\bigr)$ as a power series, then for any given $N$ we can write $f(t)$, modulo $\bigO(t^N)$, as
\begin{multline*}
\int_{\Lambda^{k-1}}(4\pi t)^{-\frac{n(k-1)}2}\Bigl(\prod_{j=1}^k r_j^{-\frac n2}\Bigr)
\int_{\R^{nk}}
\sum_{|\alpha_1+\cdots+\alpha_L|\le 2N}\bfu^{\alpha_1+\cdots+\alpha_L}
\prod_{i=1}^L \biggl(\frac{Q_{\alpha_i,\bfr,x}(\bfu)}{4t}\biggr)\\
\times e^{-\tfrac 1{4t}Q_\bfr(\bfu)}
V_k\bigl(e_x(u_k)\bigr)\,\cdots\,V_2\bigl(e_x(u_2)\bigr)\,V_1(x)\,\phi_\alpha(x,\bfu)\,d\bfu\,dx\,d\bfr.
\end{multline*}
The term $\bfu^{\alpha_1+\cdots+\alpha_L}$ will be absorbed into $\phi_\alpha(x,\bfu)$, and estimates we prove will be uniform over $\bfr\in\Lambda^{k-1}$, so it suffices to prove the following
\begin{lemma}\label{lem:reduction} Suppose that $g(t)$ takes the form
\begin{multline}\label{gform}
g(t)=(4\pi t)^{-\frac{n(k-1)}2}\Bigl(\prod_{j=1}^k r_j^{-\frac n2}\Bigr)
\int_{\R^{nk}}\,\prod_{i=1}^L \biggl(\frac{Q_{\alpha_i,\bfr,x}(\bfu)}{4t}\biggr)\, 
e^{-\tfrac 1{4t}Q_\bfr(\bfu)}\\
\times V_k\bigl(e_x(u_k)\bigr)\,\cdots\,V_2\bigl(e_x(u_2)\bigr)\,V_1(x)\,\phi(x,\bfu)\,d\bfu\,dx
\end{multline}
where $V_j\in L^\infty\cap H^m(M)$, and $Q_{\alpha_i,\bfr,x}(\bfu)$ satisfies \eqref{dominate}.
Then for $t\in(0,1]$,
\begin{equation*}
g(t)=a_0+a_1 t+\cdots+a_{m-1}t^{m-1}+r_m(t)\ts t^m,
\end{equation*}
where
\begin{equation}\label{expansion}
\begin{split}
|a_i|&\le C_{k,m}\,\Bigl(\sum_{j=1}^k\|V_j\|_{L^\infty}\Bigr)^2
\Bigl(\sum_{j=1}^k\|V_j\|_{H^i}\Bigr)^{k-2},\\
\sup_{t\in(0, 1]}|r_m(t)|&\le\,C_{k,m}\, \Bigl(\sum_{j=1}^k\|V_j\|_{L^\infty}\Bigr)^2
\Bigl(\sum_{j=1}^k\|V_j\|_{H^m}\Bigr)^{k-2}.
\end{split}
\end{equation}
\end{lemma}

In what follows, given a matrix $B$ on $\R^{n(k-1)}$ we define the quadratic form $B(\bfu)=(B\bfu)\cdot\bfu$, and given a quadratic form $B(\bfu)$ let $B$ denote the symmetric matrix that determines it. It is useful to introduce the following notation comparing matrices, via their quadratic forms, to $Q_\bfr$ or $Q_\bfr^{-1}$.

\begin{definition}
Given a family of matrices $B_{\bfr,x}$ on $\R^{n(k-1)}$, depending on parameters $\bfr\in \Lambda^{k-1}$ and $x\in U$, we write $B_{\bfr,x}\lesssim Q_\bfr$ if there is a constant $C$ such that the associated family of quadratic forms satisfies
$$
|B_{\bfr,x}(\bfu)|\le C\, Q_\bfr(\bfu)\quad\text{for all}\quad \bfr\in \Lambda^{k-1},\;\;x\in U.
$$
\end{definition}

We then express \eqref{dominate} as $\partial_x^\beta Q_{\alpha,\bfr,x}\lesssim Q_\bfr$ and $\partial_x^\beta R_{\alpha,\bfr,x}\lesssim Q_\bfr$, for all $\beta$.
The following is an immediate consequence of the definition, where $Q_\bfr^{-1/2}$ is the positive definite square root of $Q_\bfr^{-1}$,
\begin{equation}\label{compbound}
A_{\bfr,x}\lesssim Q_\bfr\;\; \Leftrightarrow\;\; 
Q_\bfr^{-1/2}A_{\bfr,x}\ts Q_\bfr^{-1/2}\lesssim I \;\;\Leftrightarrow \;\;
Q_\bfr^{-1}A_{\bfr,x}\ts Q_\bfr^{-1}\lesssim Q_\bfr^{-1}.
\end{equation}

\begin{lemma}\label{quadexpident}
Suppose that $B$ is a symmetric matrix on $\R^{n(k-1)}$, and $B(\bfu)$ is the corresponding quadratic form in $\bfu$. Then
$$
B(\bfu)e^{-\tfrac 1{4t}Q_\bfr(\bfu)}=
\Bigl(4t^2 B'_\bfr(\partial_{\bfu})+2t\trace\bigl(Q_\bfr^{-1}B\bigr)\Bigr)e^{-\tfrac 1{4t}Q_\bfr(\bfu)}
$$
where $B'_\bfr=Q_\bfr^{-1}BQ_\bfr^{-1}$, with $Q_\bfr$ the symmetric matrix associated to $Q_\bfr(\bfu)$.
\end{lemma}
\begin{proof}
Given a quadratic form $Q(v)$ with symmetric matrix $Q$ we have
$$
\partial_{v_i}\partial_{v_j}e^{-\tfrac 1{4t}Q(v)}=
\biggl(\frac{(Qv)_i(Qv)_j}{4t^2}-\frac{Q_{ij}}{2t}\biggr)e^{-\tfrac 1{4t}Q(v)},
$$
and hence for symmetric matrix $A$
$$
A(\partial_v)e^{-\tfrac 1{4t}Q(v)}=
\biggl(\frac{A(Qv)}{4t^2}-\frac{\trace(AQ)}{2t}\biggr)e^{-\tfrac 1{4t}Q(v)}.
$$
The statement of the lemma follows by taking $A=Q_\bfr^{-1}B Q_\bfr^{-1}$.
\end{proof}

\begin{proof}[Proof of Lemma \ref{lem:reduction}]
We now turn to the proof that \eqref{expansion} holds for the expression \eqref{gform}. The bounds on $a_j$ and $r_m(t)$ will follow from the proof. We divide consideration into cases.

\noindent{$\bullet$} $m=0,\;L\ge 0.\,$  This follows exactly as for the estimate \eqref{fbound} above, using \eqref{prodbound} instead of \eqref{3.12}.

\noindent{$\bullet$} $L=0,\;m\ge 1.\,$ We need to establish \eqref{expansion} for $g(t)$ of the form
\begin{multline*}
g(t)=\Bigl(\prod_{j=1}^k r_j^{-\frac n2}\Bigr)
\int_{\R^{nk}} (4\pi t)^{-\frac{n(k-1)}2}\,e^{-\tfrac 1{4t}Q_\bfr(\bfu)}\\
\times V_k\bigl(e_x(u_k)\bigr)\,\cdots\,V_2\bigl(e_x(u_2)\bigr)\,V_1(x)\,\phi(x,\bfu)\,d\bfu\,dx.
\end{multline*}
We proceed by induction on $m$, and assume the result holds at regularity $V_j\in H^{m-1}$.  We will show that when $t\in(0,1]$ and $V_j\in L^\infty_{\comp}\cap H^m(M)$ we can write
$$
g'(t)=a_1+a_2\,t+\cdots+a_{m-1} t^{m-2}+t^{m-1}r(t).
$$
This implies $g(t)$ is continuous on $0\le t\le 1$, and the expansion for $g(t)$ follows by integration.

The following identity is a simple consequence of Lemma \ref{quadexpident},
$$
\partial_te^{-\tfrac 1{4t}Q_\bfr(\bfu)}=
\Bigl(Q_\bfr^{-1}(\partial_\bfu)+\tfrac{n(k-1)}{2t}\Bigr)e^{-\tfrac 1{4t}Q_\bfr(\bfu)}.
$$
We apply this to the integrand for $g(t)$, and after integration by parts we see that $g'(t)$ equals the following:
\begin{multline*}
\Bigl(\prod_{j=1}^k r_j^{-\frac n2}\Bigr)
\int_{\R^{nk}} (4\pi t)^{-\frac{n(k-1)}2}\,e^{-\tfrac 1{4t}Q_\bfr(\bfu)}
\\
\times Q_\bfr^{-1}(\partial_\bfu)\Bigl(V_k\bigl(e_x(u_k)\bigr)\,
\cdots\,V_2\bigl(e_x(u_2)\bigr)\,V_1(x)\,\phi(x,\bfu)\Bigr)\,d\bfu\,dx
\end{multline*}
The coefficients of $Q_\bfr^{-1}$ are bounded by a fixed constant, uniformly over $\bfr\in\Lambda^{k-1}$, so we can replace $Q_\bfr^{-1}(\partial_\bfu)$ by a component of $\partial_{u_i}\partial_{u_j}$ for some $i,j$. If $i\ne j$, at most one derivative falls on a given $V_j\bigl(e_x(u_j)\bigr)$, leading to a $k$-fold product of $V_j$'s of regularity $H^{m-1}$. 
The desired expansion for $g'(t)$ follows from the induction hypothesis for regularity $m-1$.

When $i=j$, we need consider a term like
\begin{multline*}
\Bigl(\prod_{j=1}^k r_j^{-\frac n2}\Bigr)
\int_{\R^{nk}} (4\pi t)^{-\frac{n(k-1)}2}\,e^{-\tfrac 1{4t}Q_\bfr(\bfu)}\\
\times
\Bigl(\partial_{u_k}^{\ts 2}V_k\bigl(e_x(u_k)\bigr)\Bigr)\,\cdots\,V_2\bigl(e_x(u_2)\bigr)\,V_1(x)\,\phi(x,\bfu)\Bigr)\,d\bfu\,dx.
\end{multline*}
To handle this, we use Lemma \ref{utox} and integration by parts to convert one factor of 
$\partial_{u_k}$ into $\partial_x$ acting on a factor $V_j$ for $j\ne k$, and proceed as for the case $i\ne j$.

\noindent{$\bullet$} $L\ge 1,\;m\ge 0.\,$ We proceed by induction on $L$, and assume \eqref{expansion} holds for a term of the form \eqref{gform} with an $L-1$ fold product, for all integers $m\ge 0$.
We note that, by \eqref{dominate} and \eqref{compbound}, the function 
$$
\psi_{\bfr,L}(x)=\thf\trace\bigl(Q_\bfr^{-1}Q_{\bfr,\alpha_L,x}\bigr)=\thf\trace\bigl(Q_\bfr^{-1/2}Q_{\bfr,\alpha_L,x}Q_\bfr^{-1/2}\bigr)
$$
is a smooth function of $x$, with $|\partial^\alpha_x\psi_{\bfr,L}|$ uniformly bounded over $\bfr\in\Lambda^{k-1}$ and $x\in U$, for all $\alpha$.

Considering the expression \eqref{gform}, we use Lemma \ref{quadexpident} to write 
\begin{equation*}
\frac{Q_{\bfr,\alpha_L,x}(\bfu)}{4t}\,
e^{-\tfrac 1{4t}Q_\bfr(\bfu)}\\
=
\Bigl(t B_{\bfr,\alpha_L,x}(\partial_\bfu)+\psi_{\bfr,L}(x)\Bigr)e^{-\tfrac 1{4t}Q_\bfr(\bfu)}
\end{equation*}
where $B_{\bfr,\alpha_L,x}= Q_\bfr^{-1}Q_{\bfr,\alpha_L,x}Q_\bfr^{-1}$, hence $B_{\bfr,\alpha_L,x}\lesssim Q_\bfr^{-1}$ by \eqref{compbound}.
The smooth function $\psi_{\bfr,L}(x)$ can be absorbed into $\phi_\bfr(x,\bfu)$, which will denote a function in $C_{\comp}^\infty(U^k)$ with $C_{\comp}^\infty$ bounds in $(x,\bfu)$ that are uniform over $\bfr$. This term then leads to an $L-1$ fold product which is handled by the induction hypothesis.

We then need consider the commutators, for $1\le i<L$,
$$
\Bigl[Q_{\bfr,\alpha_i,x}(\bfu)\,,
B_{\bfr,\alpha_L,x}(\partial_\bfu)\Bigr]=-4\,\bfu\cdot Q_{\bfr,\alpha_i,x}
B_{\bfr,\alpha_L,x}\,\partial_\bfu-2\trace\bigl(Q_{\bfr,\alpha_i,x}
B_{\bfr,\alpha_L,x}\bigr).
$$
The trace term is a smooth, bounded function of $x$, uniformly over $r\in\Lambda^{k-1}$, since $B_{\bfr,\alpha_L,x}\lesssim Q_\bfr^{-1}$ and $Q_{\bfr,\alpha_i,x}\lesssim Q_\bfr$, and thus can be harmlessly absorbed into $\phi_\bfr(x,\bfu)$.
We also note the following,
$$
\Bigl[Q_{\bfr,\alpha_j,x}(\bfu)\,,\bfu\cdot Q_{\bfr,\alpha_i,x}
B_{\bfr,\alpha_L,x}\,\partial_\bfu\Bigr]=
-2\,\bfu\cdot Q_{\bfr,\alpha_i,x}B_{\bfr,\alpha_L,x}Q_{\bfr,\alpha_j,x}\bfu\,.
$$
The matrix $Q_{\bfr,\alpha_i,x}B_{\bfr,\alpha_L,x}Q_{\bfr,\alpha_j,x}$ is not necessarily symmetric, but the commutator involves only the symmetric part of this matrix. Additionally,
$$
\bigl|\bfu\cdot Q_{\bfr,\alpha_i,x}B_{\bfr,\alpha_L,x}Q_{\bfr,\alpha_j,x}\bfu\bigr|\le C Q_\bfr(\bfu)\,,
$$
and so the quadratic form behaves exactly like a term $Q_{\bfr,\alpha_j,x}(\bfu)$.
Also,
$$
\Bigl[\bfu\cdot Q_{\bfr,\alpha_i,x}B_{\bfr,\alpha_L,x}\,\partial_\bfu\,,B_{\bfr,\alpha_j,x}(\partial_\bfu)\Bigr]=-2\,\partial_\bfu\cdot B_{\bfr,\alpha_j,x}Q_{\bfr,\alpha_i,x}B_{\bfr,\alpha_L,x}\partial_\bfu\,,
$$
and the symmetric part of $B_{\bfr,\alpha_j,x}Q_{\bfr,\alpha_i,x}B_{\bfr,\alpha_L,x}$ is dominated by $Q_\bfr^{-1}$. Thus, for the purposes of estimates, the operator $\bfu\cdot Q_{\bfr,\alpha_i,x}B_{\bfr,\alpha_L,x}\,\partial_\bfu$ commutes with both $Q_{\bfr,\alpha_j,x}(\bfu)$ and $B_{\bfr,\alpha_j,x}(\partial_\bfu)$. Finally,
$$
\bfu\cdot Q_{\bfr,\alpha_i,x}B_{\bfr,\alpha_L,x}\,\partial_\bfu e^{-\frac{Q_\bfr(\bfu)}{4t}}=
-\frac{\bigl(\bfu\cdot Q_{\bfr,\alpha_i,x}B_{\bfr,\alpha_L,x}Q_\bfr\bfu\bigr)}{2t}\,
e^{-\frac{Q_\bfr(\bfu)}{4t}}\,,
$$
which behaves the same as multiplying by the factor $Q_{\bfr,\alpha_i,x}(\bfu)/4t$.

The end result is that we can write, up to inconsequential modifications of the $Q_{\bfr,\alpha_i,x}$,
\begin{multline*}
\prod_{i=1}^L \biggl(\frac{Q_{\bfr,\alpha_i,x}(\bfu)}{4t}\biggr)\, 
e^{-\tfrac 1{4t}Q_\bfr(\bfu)}=
t B_{\bfr,\alpha_L,x}(\partial_\bfu)\prod_{i=1}^{L-1} \biggl(\frac{Q_{\bfr,\alpha_i,x}(\bfu)}{4t}\biggr)e^{-\tfrac 1{4t}Q_\bfr(\bfu)}\\
+\phi_\bfr(\bfu,x)\prod_{i=1}^{L-1} \biggl(\frac{Q_{\bfr,\alpha_i,x}(\bfu)}{4t}\biggr)e^{-\tfrac 1{4t}Q_\bfr(\bfu)}\,.
\end{multline*}

We apply this identity to the integrand of \eqref{gform}.
The second term on the right hand side
(which is more precisely a sum of such terms) is handled by the induction hypothesis in $L$, so we continue with just the first term on the right. We integrate by parts in $\bfu$ to move the $B_{\bfr,\alpha_L,x}(\partial_\bfu)$ to act on $V\bigl(e_x(u_k)\bigr)\,\cdots\,V\bigl(e_x(u_2)\bigr)\,V(x)\,\phi_\bfr(x,\bfu)$. At this point, the only estimate we use on $B_{\bfr,\alpha_L,x}$ is that it is a bounded matrix, together with all derivatives in $x$, which follows since $B_{\bfr,\alpha_L,x}\lesssim Q_\bfr^{-1}\lesssim I\,,$ similarly for its derivatives in $x$. Thus, the coefficients of $B_{\bfr,\alpha_L,x}$ can be absorbed into $\phi_\bfr(x,\bfu)$, leading to the term
\begin{multline*}
t\,(4\pi t)^{-\frac{n(k-1)}2}\Bigl(\prod_{j=1}^k r_j^{-\frac n2}\Bigr)
\int_{\R^{nk}}\,\prod_{i=1}^{L-1} \biggl(\frac{Q_{\bfr,\alpha_i,x}(\bfu)}{4t}\biggr)\, 
e^{-\tfrac 1{4t}Q_\bfr(\bfu)}\\
\times \partial_\bfu^{\ts 2}\Bigl(V\bigl(e_x(u_k)\bigr)\,\cdots\,V\bigl(e_x(u_2)\bigr)\,V(x)\,\phi_\bfr(x,\bfu)\Bigr)\,d\bfu\,dx\,.
\end{multline*}

This is handled as above, using Lemma \ref{utox} and the result for $m-1$ and $L-1$. The only difference is that when we convert a factor of $\partial_{u_k}$ into $\partial_x$, in addition to acting on the other factors of $V_j$ the operator $\partial_x$ can also act on the $Q_{\bfr,\alpha_i,x}$, which is harmless by \eqref{dominate}.
\end{proof}


\nocite{*}
\bibliographystyle{amsplain}
\bibliography{heat_trace_arx}

\providecommand{\bysame}{\leavevmode\hbox to3em{\hrulefill}\thinspace}
\providecommand{\MR}{\relax\ifhmode\unskip\space\fi MR }
\providecommand{\MRhref}[2]{%
  \href{http://www.ams.org/mathscinet-getitem?mr=#1}{#2}
}
\providecommand{\href}[2]{#2}
\begin{thebibliography}{10}

\bibitem{BGM}
Marcel Berger, Paul Gauduchon, and Edmond Mazet, \emph{Le spectre d'une
  vari\'et\'e riemannienne}, Lecture Notes in Mathematics, Vol. 194,
  Springer-Verlag, Berlin-New York, 1971. \MR{0282313}

\bibitem{Bishop}
Richard~L. Bishop and Richard~J. Crittenden, \emph{Geometry of manifolds}, Pure
  and Applied Mathematics, Vol. XV, Academic Press, New York-London, 1964.
  \MR{0169148}

\bibitem{Br}
Jochen Br{\"u}ning, \emph{On the compactness of isospectral potentials}, Comm.
  Partial Differential Equations \textbf{9} (1984), no.~7, 687--698. \MR{745021
  (85h:58170)}

\bibitem{cheng}
Siu~Yuen Cheng, Peter Li, and Shing~Tung Yau, \emph{On the upper estimate of
  the heat kernel of a complete {R}iemannian manifold}, Amer. J. Math.
  \textbf{103} (1981), no.~5, 1021--1063. \MR{630777}

\bibitem{chow}
Bennett Chow, Sun-Chin Chu, David Glickenstein, Christine Guenther, James
  Isenberg, Tom Ivey, Dan Knopf, Peng Lu, Feng Luo, and Lei Ni, \emph{The
  {R}icci flow: techniques and applications. {P}art {III}. {G}eometric-analytic
  aspects}, Mathematical Surveys and Monographs, vol. 163, American
  Mathematical Society, Providence, RI, 2010. \MR{2604955}

\bibitem{TC}
T.~Christiansen, \emph{Schr\"odinger operators with complex-valued potentials
  and no resonances}, Duke Math. J. \textbf{133} (2006), no.~2, 313--323.
  \MR{2225694 (2007h:35246)}

\bibitem{CdV}
Yves Colin~de Verdi{\`e}re, \emph{Une formule de traces pour l'op\'erateur de
  {S}chr\"odinger dans {${\bf R}\sp{3}$}}, Ann. Sci. \'Ecole Norm. Sup. (4)
  \textbf{14} (1981), no.~1, 27--39. \MR{618729 (82g:35088)}

\bibitem{Dod}
Jozef Dodziuk, \emph{Maximum principle for parabolic inequalities and the heat
  flow on open manifolds}, Indiana Univ. Math. J. \textbf{32} (1983), no.~5,
  703--716. \MR{711862}

\bibitem{Don}
Harold Donnelly, \emph{Compactness of isospectral potentials}, Trans. Amer.
  Math. Soc. \textbf{357} (2005), no.~5, 1717--1730 (electronic). \MR{2115073
  (2006d:58032)}

\bibitem{MM}
H.~P. McKean and P.~van Moerbeke, \emph{The spectrum of {H}ill's equation},
  Invent. Math. \textbf{30} (1975), no.~3, 217--274. \MR{0397076 (53 \#936)}

\bibitem{MP}
S.~Minakshisundaram and {\AA}.~Pleijel, \emph{Some properties of the
  eigenfunctions of the {L}aplace-operator on {R}iemannian manifolds}, Canadian
  J. Math. \textbf{1} (1949), 242--256. \MR{0031145}

\bibitem{rosenberg}
Steven Rosenberg, \emph{The {L}aplacian on a {R}iemannian manifold}, London
  Mathematical Society Student Texts, vol.~31, Cambridge University Press,
  Cambridge, 1997, An introduction to analysis on manifolds. \MR{1462892}

\bibitem{SaBZw}
Ant{\^o}nio S{\'a}~Barreto and Maciej Zworski, \emph{Existence of resonances in
  three dimensions}, Comm. Math. Phys. \textbf{173} (1995), no.~2, 401--415.
  \MR{1355631}

\bibitem{SmZw}
Hart~F. Smith and Maciej Zworski, \emph{Heat traces and existence of scattering
  resonances for bounded potentials}, Ann. Inst. Fourier (Grenoble) \textbf{66}
  (2016), no.~2, 455--475. \MR{3477881}

\bibitem{pde}
Michael~E. Taylor, \emph{Partial differential equations {III}. {N}onlinear
  equations}, second ed., Applied Mathematical Sciences, vol. 117, Springer,
  New York, 2011. \MR{2744149 (2011m:35003)}

\end{thebibliography}


\end{document}